\documentclass[11pt,reqno]{amsart}
\usepackage{amssymb,amsmath}
\newcommand{\Z}{\mathbb{Z}}
\newcommand{\R}{\mathbb{R}}
\newcommand{\C}{\mathbb{C}}
\newcommand{\He}{\operatorname{H}}
\newcommand{\Het}{\mathcal{H}}
\newcommand\D{\operatorname{D}}
\newcommand\Ha{\mathbb{H}}
\renewcommand\theta{\vartheta}
\newtheorem{theorem}{Theorem}[section]
\newtheorem{lemma}[theorem]{Lemma}
\newtheorem{remark}[theorem]{Remark}
\numberwithin{equation}{section}

\title{Rank-Crank type PDE's for higher level Appell functions}
\date{\today}
\author{Sander Zwegers}
\address{School of Mathematical Sciences, University College Dublin, Belfield, Dublin 4, Ireland}
\email{sander.zwegers@ucd.ie}
\begin{document}
\begin{abstract}
In this paper we consider level $l$ Appell functions, and find a partial differential equation for all odd $l$. For $l=3$ this recovers the Rank-Crank PDE, found by Atkin and Garvan, and for $l=5$ we get a similar PDE found by Garvan.
\end{abstract}
\maketitle
\section{Introduction and statement of results}
Dyson in \cite{Dy} introduced the \textit{rank} of a partition, to explain the first two of the three Ramanujan-congruences
\begin{equation} \label{Ramanujan}
\begin{split}
p(5n+4) & \equiv 0 \pmod 5,\\
p(7n+5) & \equiv 0 \pmod 7, \\
p(11n+6) & \equiv 0 \pmod {11}.
\end{split}
\end{equation}
Here $p(n)$ denotes the number of partitions of $n$. He defined the rank of a partition as the largest part minus the number of its parts and conjectured that the partitions of $5n+4$ (resp.\ $7n+5$) form
$5$ (resp.\ $7$) groups of equal size when sorted by their ranks
modulo $5$ (resp.\ $7$). This was later proven by  
Atkin and Swinnerton-Dyer in \cite{AS}. We are interested here in the generating function
\begin{equation*}
R(w;q):= \sum_{\lambda}
w^{\operatorname{rank}(\lambda)} q^{||\lambda ||}
= \frac{(1-w)}{(q)_{\infty}} \sum_{n \in \Z}
\frac{(-1)^n q^{\frac{n}{2}(3n+1) }}{1-wq^n},
\end{equation*}
where $(q)_{\infty}:=
\prod_{n=1}^\infty (1-q^n)$. In the first sum the $\lambda$ run over all partitions, $\operatorname{rank}(\lambda)$ denotes the rank of $\lambda$ and $||\lambda ||$ denotes the size of the partition (the sum of all the parts).

Another partition statistic is the so called \textit{crank} of a partition. For the generating function we have
\begin{equation*}
\begin{split}
C(w;q) := \sum_{\lambda}
w^{\operatorname{crank}(\lambda)} q^{||\lambda ||} &= \prod_{n=1}^{\infty} \frac{  \left( 1-q^n\right) }{  \left( 1-wq^n\right)\left(1-w^{-1} q^n \right)  }\\
&= (1-w) \frac{(q)_\infty^2}{\sum_{n\in\Z} (-1)^n q^{\frac{1}{2}n(n-1)} w^n}.
\end{split}
\end{equation*}
The crank was introduced by Andrews and Garvan in \cite{AnG} to explain the Ramanujan congruence (\ref{Ramanujan}) with modulus $11$. 

In the setting of Jacobi forms it is more natural to consider the following modified rank and crank generating functions
\begin{equation*}
\begin{split}
\mathcal{R}(z;\tau)&:=\frac{w^{1/2} q^{-1/24}}{1-w} R(w,q),\\
\mathcal{C}(z;\tau)&:=\frac{w^{1/2} q^{-1/24}}{1-w} C(w,q).
\end{split}
\end{equation*}
Here we use $w=\exp(2\pi iz)$ and $q=\exp(2\pi i\tau)$, with $z\in\C$ and $\tau$ in the complex upper half plane $\Ha$.
\begin{remark}
$\mathcal{C}$ is a meromorphic Jacobi form of weight $1/2$ and index $-1/2$ and in \cite{Zw} it is shown that $\mathcal{R}$ is mock Jacobi form of weight $1/2$ and index $-3/2$.
\end{remark}

The two (modified) generating functions are related by a partial differential equation, which we will refer to as the Rank-Crank PDE.

\begin{theorem}[see \cite{AG}]
If we define the heat operator $\He$ by
\begin{equation*}
\He := \frac{3}{\pi i} \frac{\partial}{\partial \tau} + \frac{1}{(2\pi i)^2} \frac{\partial^2}{\partial z^2}, 
\end{equation*}
then
\begin{equation*}
\He \mathcal{R} = 2 \eta^2 \mathcal{C}^3,
\end{equation*}
where $\eta$ is the Dedekind $\eta$-function, given by $\eta(\tau) = q^{1/24} (q)_\infty$.
\end{theorem}
Note that the identity found in \cite{AG} is slightly different, because they use a different normalization. However, the two are easily seen to be equivalent. In \cite{BZ} it is explained how the Rank-Crank PDE arises naturally in the setting of certain non-holomorphic Jacobi forms and a generalization is given to  partial differential equations for an infinite family of related functions.

The method in \cite{BZ} works only in certain special cases and no results are found for the level $l$ Appell functions
\begin{equation}\label{defA}
A_l (z;\tau):= w^{l/2} \sum_{n\in\Z} \frac{(-1)^{ln} q^{\frac{l}{2}n(n+1)}}{1-wq^n} \qquad l\in\Z_{>0},
\end{equation}
for values of $l$ higher than 3.

Garvan (\cite{Ga}), however, found the following PDE for a level 5 Appell function
\begin{theorem}[Garvan]\label{Garvan}
Let
\begin{equation*}
G_5 (z;\tau) := \frac{A_5(z;\tau)}{\eta(\tau)^3} ,
\end{equation*}
and define the heat operator
\begin{equation*}
\He := \frac{5}{\pi i} \frac{\partial}{\partial \tau} + \frac{1}{(2\pi i)^2} \frac{\partial^2}{\partial z^2}, 
\end{equation*}
then
\begin{equation*}
\left(\He^2 -E_4\right) G_5 = 24 \eta^2 \mathcal{C}^5,
\end{equation*}
where $E_4$ is the usual Eisenstein series
\begin{equation*}
E_4(\tau) = 1 + 240 \sum_{n=1}^\infty \Bigl(\sum_{d |n} d^3 \Bigr) q^n.
\end{equation*}
\end{theorem}
Note that the identity found by Garvan is slightly different, because he uses a different normalization. However, the two are easily seen to be equivalent.

This theorem is a special case of the following
\begin{theorem}\label{main}
Let $l$ be an odd positive integer. Define
\begin{equation*}
\begin{split}
\Het_k &:= \frac{l}{\pi i} \frac{\partial}{\partial \tau} + \frac{1}{(2\pi i)^2} \frac{\partial^2}{\partial z^2}-\frac{l(2k-1)}{12} E_2,\\
\Het^k &:= \Het_{2k-1} \Het_{2k-3} \cdots \Het_{3} \Het_{1},
\end{split}
\end{equation*}
where $E_2 (\tau) = 1 -24 \sum_{n=1}^\infty \Bigl(\sum_{d |n} d \Bigr) q^n$ is the usual Eisenstein series in weight 2. Then there exist holomorphic modular forms $f_j$ ($j=0,2,4,\ldots ,l-1$) on $\operatorname{SL}_2(\Z)$ of weight $j$, such that
\begin{equation*}
\sum_{k=0}^{(l-1)/2} f_{l-2k-1} \Het^k A_l = (l-1)! f_0 \eta^l \mathcal{C}^l.
\end{equation*}
\end{theorem}
\begin{remark}
In the proof of the theorem we will see an explicit construction for the $f_j$'s for given $l$.
\end{remark}
In the next section we will proof Theorem \ref{main} and in section 3 we will look at the first few cases and in particular we'll see that the theorem for $l=5$ is equivalent to Theorem \ref{Garvan}.

\section*{Acknowledgements}
The author wishes to thank Frank Garvan and Karl Mahlburg for helpful discussions.

\section{Proof of Theorem \ref{main}}
Throughout we assume that $l$ is an odd positive integer. We (trivially) have
\begin{equation}\label{sh}
A_l(z+1;\tau)= - A_l(z;\tau),
\end{equation}
and if we replace $z$ by $z+\tau$ and $n$ by $n-1$ in \eqref{defA} we find
\begin{equation*}
e^{-2\pi i lz-\pi i l\tau} A_l(z+\tau;\tau) = - w^{-l/2} \sum_{n\in\Z} \frac{(-1)^{n} q^{\frac{l}{2}n(n-1)}}{1-wq^n},
\end{equation*}
and so
\begin{equation}\label{shift}
\begin{split}
A_l(z;\tau)&+e^{-2\pi i lz-\pi i l\tau} A_l(z+\tau;\tau)= - w^{-l/2} \sum_{n\in\Z} \frac{(-1)^{n} q^{\frac{l}{2}n(n-1)}}{1-wq^n} \left(1-w^l q^{ln} \right)\\
&= - w^{-l/2} \sum_{n\in\Z} (-1)^{n} q^{\frac{l}{2}n(n-1)} \sum_{r=0}^{l-1} w^r q^{nr}\\
&= - \sum_{r=0}^{l-1} w^{r-l/2} q^{-\frac{1}{2l}(r-l/2)^2} \sum_{n\in\Z} (-1)^n q^{\frac{l}{2}(n-1/2+r/l)^2}\\
&= - \sum_{r=0}^{l-1} e^{2\pi i(r-l/2)z -\frac{\pi i}{l} (r-l/2)^2 \tau} \theta_{l,r}(\tau),
\end{split}
\end{equation}
with
\begin{equation*}
\theta_{l,r}(\tau) := \sum_{n\in\Z} (-1)^n q^{\frac{l}{2}(n-1/2+r/l)^2}.
\end{equation*}
It is easy to check that 
\begin{equation*}
\left(\frac{l}{\pi i} \frac{\partial}{\partial \tau} + \frac{1}{(2\pi i)^2} \frac{\partial^2}{\partial z^2}\right) e^{2\pi i(r-l/2)z -\frac{\pi i}{l} (r-l/2)^2 \tau}=0,
\end{equation*}
and that for functions $F:\C \times \Ha \to \C$
\begin{equation*}
\begin{split}
\Het_k \left( F(z+1;\tau) \right)&= \left( \Het_k F\right)(z+1;\tau),\\
\Het_k \left( e^{-2\pi ilz-\pi il\tau} F(z+\tau;\tau)\right)&=
e^{-2\pi ilz-\pi il\tau}\left(\Het_k F\right)(z+\tau;\tau),
\end{split}
\end{equation*}
with $\Het_k$ as in the theorem. Hence we get from applying $\Het_{1}$ to equations \eqref{sh} and \eqref{shift}
\begin{equation*}
(\Het_{1} A_l)(z+1;\tau)= - (\Het_{1} A_l)(z;\tau),
\end{equation*}
and
\begin{equation*}
\begin{split} 
(\Het_{1} A_l)(z;\tau)\ +\ &e^{-2\pi i lz-\pi i l\tau} (\Het_{1} A_l)(z+\tau;\tau)\\
&=- 2l\sum_{r=0}^{l-1} e^{2\pi i(r-l/2)z -\frac{\pi i}{l} (r-l/2)^2 \tau} (\D_{1/2} \theta_{l,r})(\tau),
\end{split}
\end{equation*}
with the operator $\D_k$ defined by
\begin{equation*}
\D_k := \frac{1}{2\pi i} \frac{\partial}{\partial \tau} - \frac{k}{12} E_2.
\end{equation*}
If we now apply $\Het_3$, $\Het_5$, \ldots, upto $\Het_{2k-1}$ we find
\begin{equation}\label{tr1}
(\Het^k A_l)(z+1;\tau)= - (\Het^k A_l)(z;\tau),
\end{equation}
and 
\begin{equation}\label{tr2}
\begin{split} 
(\Het^k A_l)(z;\tau)+&e^{-2\pi i lz-\pi i l\tau} (\Het^k A_l)(z+\tau;\tau)\\
&=- (2l)^k \sum_{r=0}^{l-1} e^{2\pi i(r-l/2)z -\frac{\pi i}{l} (r-l/2)^2 \tau} (\D^k \theta_{l,r})(\tau),
\end{split}
\end{equation}
with
\begin{equation*}
\D^k := \D_{2k-3/2} \D_{2k-7/2} \cdots \D_{5/2} \D_{1/2}.
\end{equation*}
We need the following
\begin{lemma}\label{lem}
Let $l$ be an odd positive integer, then there exist holomorphic modular forms $F_j$ ($j=0,2,4,\ldots ,l-1$) on $\operatorname{SL}_2(\Z)$ of weight $j$, such that
\begin{equation}\label{eqlem}
\sum_{k=0}^{(l-1)/2} F_{l-2k-1} \D^k \theta_{l,r} = 0
\end{equation}
for all $r\in\Z$.
\end{lemma}
If we now define
\begin{equation*}
P= \sum_{k=0}^{(l-1)/2} f_{l-2k-1} \Het^k A_l,
\end{equation*}
with $f_{l-2k-1} = (2l)^{-k} F_{l-2k-1}$ and $F_j$ as in the lemma, then we see from equation \eqref{tr1} and \eqref{tr2}
\begin{equation}\label{trP}
P(z+1;\tau) = e^{-2\pi i lz-\pi i l\tau} P(z+\tau;\tau)= -P(z;\tau).
\end{equation}
Now consider the Jacobi theta function
\begin{equation*}
\begin{split}
\theta(z;\tau)&:= \sum_{n\in\Z} (-1)^n w^{n+1/2} q^{\frac{1}{2}(n+1/2)^2} \\
&= w^{1/2} q^{1/8} \prod_{n=1}^\infty (1-q^n)(1-wq^n)(1-w^{-1}q^{n-1})\\
&= -\frac{\eta(\tau)^2}{\mathcal{C}(z;\tau)}.
\end{split}
\end{equation*}
This function satisfies
\begin{equation}\label{trt}
  \theta(z+1;\tau)= e^{2\pi iz +\pi i\tau} \theta(z+\tau;\tau) = -\theta(z;\tau),
\end{equation}
$z\mapsto \theta(z;\tau)$ has simple zeros in $\Z \tau +\Z$ and 
\begin{equation}\label{nult}
\frac{1}{2\pi i} \left. \frac{\partial}{\partial z} \right|_{z=0} \theta(z;\tau) = \eta(\tau)^3.
\end{equation}
Since the  poles of $z\mapsto A_l(z;\tau)$ are simple poles in $\Z\tau+\Z$, the function $z\mapsto P(z;\tau)$ has poles of order $l$ in $\Z\tau+\Z$, and so the function 
\begin{equation*}
p(z;\tau):= \theta(z;\tau)^l P(z;\tau),
\end{equation*}
is a holomorphic function as a function of $z$.  Using \eqref{trP} and \eqref{trt} we find that
\begin{equation*}
p(z+1;\tau)=p(z+\tau;\tau)=p(z;\tau),
\end{equation*}
from which we get that $p$ is constant (as a function of $z$). To determine the constant, we consider the behaviour for $z\to 0$. From \eqref{defA} we easily see that for $z \to 0$
\begin{equation*}
A_l(z;\tau) = -\frac{1}{2\pi i}\frac{1}{z} + \mathcal{O}(1),
\end{equation*}
and so 
\begin{equation*}
P(z;\tau) = -f_0(\tau) \frac{(l-1)!}{(2\pi i)^l} \frac{1}{z^l} + \mathcal{O}\left(\frac{1}{z^{l-1}}\right).
\end{equation*}
Combining this with \eqref{nult} we see
\begin{equation*}
p(z;\tau) = -f_0(\tau) (l-1)!\eta(\tau)^{3l},
\end{equation*}
and so
\begin{equation*}
P(z;\tau) = -f_0(\tau) (l-1)! \frac{\eta(\tau)^{3l}}{\theta(z;\tau)^l}= (l-1)! f_0(\tau) \eta(\tau)^l \mathcal{C}(z;\tau)^l,
\end{equation*}
which finishes the proof.

\begin{proof}[Proof of Lemma \ref{lem}]
Throughout, let $l$ be an odd integer. Because of the trivial relations
\begin{equation*}
\begin{split}
\theta_{l,r+l} &= -\theta_{l,r}\\
\theta_{l,-r} &= -\theta_{l,r}
\end{split}
\end{equation*}
it suffices to consider $\theta_{l,r}$ for $r=1,2,\ldots, (l-1)/2$. Define
\begin{equation*}
\Theta_l= \begin{pmatrix} \theta_{l,1}\\ \theta_{l,2} \\ \vdots \\ \theta_{l,(l-1)/2} \end{pmatrix},
\end{equation*}
then $\Theta_l$ transforms as a (vector-valued) modular form of weight 1/2 on the full modular group $\operatorname{SL}_2(\Z)$:
\begin{equation*}
\begin{split}
\Theta_l (\tau+1) &= \operatorname{diag} \left( \zeta_{8l}^{(l-2j)^2} \right)_{1\leq j \leq (l-1)/2} \Theta_l(\tau),\\
\Theta_l (-1/\tau) &= (-1)^{(l+1)/2} \sqrt{\tau/li}\ (2\sin 2\pi rk/l)_{1\leq r,k \leq (l-1)/2}\ \Theta_l (\tau).
\end{split}
\end{equation*}
Using
\begin{equation*}
E_2\left(\frac{a\tau+b}{c\tau +d}\right) = (c\tau+d)^2 E_2 (\tau) +\frac{6}{\pi i}c(c\tau+d)\qquad\text{for}\  \left(\begin{smallmatrix}a&b\\c&d\end{smallmatrix}\right)\in \operatorname{SL}_2(\Z),
\end{equation*}
we can easily verify that
\begin{equation*}
\D_k \left( (c\tau+d)^{-k} f\left( \frac{a\tau+b}{c\tau+d}\right) \right) = (c\tau+d)^{-k-2} (\D_k f) \left(\frac{a\tau+b}{c\tau+d}\right),
\end{equation*}
and so
\begin{equation*}
\D^k \left( (c\tau+d)^{-1/2} \Theta_l \left(\frac{a\tau+b}{c\tau+d}\right)\right) = (c\tau+d)^{-2k-1/2} \left( \D^k \Theta_l\right) \left(\frac{a\tau+b}{c\tau+d}\right).
\end{equation*}

Now define the $(l-1)/2 \times (l-1)/2$-matrix
\begin{equation*}
T_l = \begin{pmatrix} \Theta_l & \D^1\Theta_l & \D^2 \Theta_l & \cdots & \D^{(l-3)/2} \Theta_l \end{pmatrix},
\end{equation*}
then $T_l$ transforms as a (matrix-valued) modular form on the full modular group $\operatorname{SL}_2(\Z)$:
\begin{equation}\label{modT}
\begin{split}
T_l (\tau+1) &= \operatorname{diag} \left( \zeta_{8l}^{(l-2j)^2} \right)_{1\leq j \leq (l-1)/2} T_l(\tau),\\
T_l (-1/\tau) &= (-1)^{(l+1)/2} \sqrt{\tau/li}\ (2\sin 2\pi rk/l)_{1\leq r,k \leq (l-1)/2}\ T_l (\tau)\operatorname{diag} \left( \tau^{2j-2} \right)_{1\leq j \leq (l-1)/2}.
\end{split}
\end{equation}
From this we see that
\begin{equation*}
\begin{split}
\det (T_l (\tau+1)) &= \zeta_{24}^{(l-1)(l-2)/2} \det(T_l(\tau)),\\
\det (T_l (-1/\tau)) &= (-i\tau)^{(l-1)(l-2)/4} \det(T_l(\tau)),
\end{split}
\end{equation*}
and so $\det(T_l)$ is a multiple of $\eta^{(l-1)(l-2)/2}$. We determine what that multiple is by looking at the lowest order terms: \\
First observe that by doing elementary column operations we get
\begin{equation*}
\det(T_l(\tau)) = \det \begin{pmatrix} \Theta_l & \partial_{\tau} \Theta_l & \partial_{\tau}^2 \Theta_l & \cdots & \partial_{\tau}^{(l-3)/2} \Theta_l \end{pmatrix},
\end{equation*}
with $\partial_{\tau} := \frac{1}{2\pi i} \frac{\partial}{\partial \tau}$.

For $1\leq r \leq (l-1)/2$ we have
\begin{equation*}
\theta_{l,r} (\tau) = q^{(l-2r)^2/8l} \left( 1+ \mathcal{O} (q) \right),
\end{equation*}
so 
\begin{equation*}
\begin{split}
&\begin{pmatrix} \Theta_l & \partial_{\tau} \Theta_l & \partial_{\tau}^2 \Theta_l & \cdots & \partial_{\tau}^{(l-3)/2} \Theta_l \end{pmatrix} \\
&\quad = \operatorname{diag} \left( q^{(l-2i)^2/8l}\right)_{1\leq i \leq (l-1)/2}\cdot \left( \Bigl(\frac{(l-2i)^2}{8l} \Bigr)^{j-1} +\mathcal{O}(q)\right)_{1\leq i,j \leq (l-1)/2},\\
\det & \begin{pmatrix} \Theta_l & \partial_{\tau} \Theta_l & \partial_{\tau}^2 \Theta_l & \cdots & \partial_{\tau}^{(l-3)/2} \Theta_l \end{pmatrix} \\
&\quad =q^{(l-1)(l-2)/48} \left( \det (B)+\mathcal{O}(q) \right),
\end{split}
\end{equation*}
and hence
\begin{equation}\label{det}
\det(T_l(\tau)) = \det(B)\ \eta(\tau)^{(l-1)(l-2)/2}, 
\end{equation}
with
\begin{equation*}
B_{ij} =\Bigl(\frac{(l-2i)^2}{8l} \Bigr)^{j-1} \qquad \text{for} \qquad {1\leq i,j \leq (l-1)/2}.
\end{equation*}
$B$ is a Vandermonde matrix: an $m\times n$ matrix $V$, such that $V_{ij}= \alpha_i^{j-1}$ with $\alpha_i\in\R$. Since a square Vandermonde matrix is invertible if and only if the $\alpha_i$ are distinct, we see that $B$ is invertible.  From \eqref{det} and the fact that $\eta$ has no zeros on $\Ha$ we then get that $T_l(\tau)$ is invertible for all $\tau\in\Ha$. We can rewrite the condition that \eqref{eqlem} holds for $1\leq r\leq (l-1)/2$ as
\begin{equation*}
T_l \begin{pmatrix} F_{l-1}\\ F_{l-3}\\ \vdots \\ F_2 \end{pmatrix} + F_0 \D^{(l-1)/2} \Theta_l =0.
\end{equation*}
If we take $F_0=1$ we get the other $F_j$'s by inverting $T_l$
\begin{equation}\label{defT}
\begin{split}
\begin{pmatrix} F_{l-1}\\ F_{l-3}\\ \vdots \\ F_2 \end{pmatrix} = -T_l^{-1} \D^{(l-1)/2} \Theta_l.
\end{split}
\end{equation}
What remains to be shown is that the $F_j$ found this way are holomorphic modular forms of weight $j$ on $\operatorname{SL}_2 (\Z)$. The modular transformation properties follow easy from \eqref{modT} and those of $\D^{(l-1)/2} \Theta_l$
\begin{equation*}
\begin{split}
\D^{(l-1)/2}\Theta_l (\tau+1) &= \operatorname{diag} \left( \zeta_{8l}^{(l-2j)^2} \right)_{1\leq j \leq (l-1)/2} \D^{(l-1)/2}\Theta_l(\tau),\\
\D^{(l-1)/2}\Theta_l (-1/\tau) &= (-1)^{(l+1)/2} \sqrt{\tau/li}\ \tau^{l-1}\ (2\sin 2\pi rk/l)_{1\leq r,k \leq (l-1)/2}\ \D^{(l-1)/2}\Theta_l (\tau).
\end{split}
\end{equation*}
Since $\det T_l$ has no zeros on $\Ha$ we get that $F_j$ is a holomorphic function on $\Ha$. That it also doesn't have a pole at infinity follows from
\begin{equation*}
\begin{split}
\D^{(l-1)/2} \Theta_l &= \operatorname{diag} \left( q^{(l-2i)^2/8l}\right)_{1\leq i \leq (l-1)/2}\cdot \begin{pmatrix} \mathcal{O}(1)\\ \mathcal{O}(1)\\ \vdots \\ \mathcal{O}(1)\end{pmatrix},\\
T_l &= \operatorname{diag} \left( q^{(l-2i)^2/8l}\right)_{1\leq i \leq (l-1)/2}\cdot \left( C_{ij} + \mathcal{O}(q) \right)_{1\leq i,j \leq (l-1)/2},
\end{split}
\end{equation*}
for some $(l-1)/2\times (l-1)/2$-matrix $C$, with
\begin{equation*}
\det (C) = \det (B) \not= 0.
\end{equation*}

\end{proof}

\section{Some examples}
From \eqref{defT} we can calculate the first few coefficients in the Fourier expansion of the $F_j$'s and since they are holomorphic modular forms on $\operatorname{SL}_2(\Z)$, that means that we can easily identify them.

For $l=3$, we have
\begin{equation*}
\Theta_l= \begin{pmatrix} \theta_{3,1} \end{pmatrix} = \begin{pmatrix} \eta \end{pmatrix}.
\end{equation*}
Using
\begin{equation*}
\D_{k/2} \left( \eta^k \right) =0,
\end{equation*}
which follows from
\begin{equation*}
E_2 = \frac{12}{\pi i} \frac{\eta'}{\eta},
\end{equation*}
we see
\begin{equation*}
D_{1/2} \Theta_l =0,
\end{equation*}
and so we find
\begin{equation*}
F_0(\tau) =1 \qquad \text{and} \qquad F_2(\tau)=0,
\end{equation*}
and 
\begin{equation*}
f_0(\tau) =1/6 \qquad \text{and} \qquad f_2(\tau)=0.
\end{equation*}
If we put this into Theorem \ref{main} and multiply by 6 we get
\begin{equation*}
\Het_{1} A_3 =2 \eta^3 \mathcal{C}^3.
\end{equation*}
Using
\begin{equation*}
\mathcal{R}(z;\tau) = \frac{A_3(z;\tau)}{\eta(\tau)} +e^{\pi iz -\pi i\tau/12},
\end{equation*}
we see
\begin{equation*}
\Het_{1/2} \mathcal{R} = \Het_{1/2} \left( \frac{A_3}{\eta} \right)= \frac{\Het_{1} A_3}{\eta} + 6 A_3 \D_{-1/2} \left( \frac{1}{\eta}\right)= 2\eta^2 \mathcal{C}^3,
\end{equation*}
which is the Rank-Crank PDE.

For $l=5$, we find from \eqref{defT} ($F_0=1$)
\begin{equation*}
\begin{split}
F_4 (\tau) &= -\frac{11}{3600} -\frac{11}{15} q + \mathcal{O}(q^2),\\
F_2 (\tau) &= \mathcal{O}(q^2),
\end{split}
\end{equation*}
and hence we can identify them as
\begin{equation*}
F_4 = -\frac{11}{3600} E_4 \qquad \text{and} \qquad F_2=0.
\end{equation*}
So
\begin{equation*}
f_0 = \frac{1}{100}, \qquad f_2 =0 \qquad \text{and}\qquad f_4 = -\frac{11}{3600} E_4.
\end{equation*}
If we put this into Theorem \ref{main} and multiply by 100 we get
\begin{equation*}
\left(\Het_{3} \Het_{1} -\frac{11}{36}E_4\right) A_5 =24 \eta^5 \mathcal{C}^5.
\end{equation*}
We now rewrite this in terms of $G_5$:
\begin{equation*}
\begin{split}
\Het_{1} A_5 &= \Het_{1} \left( \eta^3 G_5\right) \\
&= 10 \left(\D_{3/2} \eta^3\right) G_5 + \eta^3 \Het_{-1/2} G_5 = \eta^3 \Het_{-1/2} G_5,\\
\Het_{3} \Het_{1} A_5 &= \Het_{3} \left( \eta^3 \Het_{-1/2} G_5 \right)\\
&= 10 \left(\D_{3/2} \eta^3\right) \Het_{-1/2} G_5 + \eta^3 \Het_{3/2} \Het_{-1/2} G_5 =  \eta^3 \Het_{3/2} \Het_{-1/2} G_5,
\end{split}
\end{equation*}
and so we get 
\begin{equation*}
\left(\Het_{3/2} \Het_{-1/2} -\frac{11}{36}E_4\right) G_5 =24 \eta^2 \mathcal{C}^5.
\end{equation*}
Using
\begin{equation*}
\Het_{3/2} \Het_{-1/2} = \He^2 +\frac{25}{3} \left( \frac{1}{2\pi i} E_2' -\frac{1}{12} E_2^2 \right)
\end{equation*}
and
\begin{equation*}
\frac{1}{2\pi i} E_2' -\frac{1}{12} E_2^2 = -\frac{1}{12} E_4
\end{equation*}
we see that is equivalent to the statement of Theorem \ref{Garvan}.

For $l=7$ we find
\begin{equation*}
F_6 = \frac{85}{74088} E_6,\qquad F_4= -\frac{5}{252} E_4, \qquad F_2=0 \qquad \text{and}\qquad F_0=1.
\end{equation*}

For $l=9$
\begin{equation*}
F_8 = -\frac{253}{559872} E_8,\qquad F_6 = \frac{53}{5832} E_6,\qquad F_4= -\frac{13}{216} E_4, \qquad F_2=0, \qquad F_0=1.
\end{equation*}

For $l=11$
\begin{equation*}
\begin{split}
F_{10} &= \frac{7888}{39135393} E_{10},\qquad F_8 = -\frac{6151}{1724976} E_8,\qquad F_6 = \frac{295}{8712} E_6,\\
F_4 &= -\frac{53}{396} E_4, \qquad F_2=0, \qquad F_0=1.
\end{split}
\end{equation*}

And for $l=13$
\begin{equation*}
\begin{split}
F_{12} &= -\frac{1462986875}{14412774445056} E_{12}+\frac{170060275}{5683867488} \Delta, \qquad F_{10} = \frac{377735}{296120448} E_{10},\\
F_8 &= -\frac{621665}{45556992} E_8,\qquad F_6 = \frac{3281}{36504} E_6, \qquad F_4 = -\frac{459}{1872} E_4, \qquad F_2=0, \qquad F_0=1.
\end{split}
\end{equation*}

\end{document}